\documentclass[11pt]{article}
\usepackage{authblk}
\usepackage{fullpage}
\usepackage{float}
\usepackage{amsmath}
\usepackage{amsthm}
\usepackage{amssymb}
\usepackage{bbm}
\usepackage{graphicx}
\usepackage[utf8]{inputenc}
\usepackage[unicode]{hyperref}

\usepackage{tikz}
\usepackage{enumerate}
\usepackage{algorithm}
\usepackage{algpseudocode}
\usepackage{breakcites}
\usepackage{filecontents}
\usepackage{natbib}

\newtheorem{assumption}{Assumption} 
\newtheorem{remark}{Remark} 
\newtheorem{theorem}{Theorem}
\newtheorem{lemma}{Lemma} 
\newtheorem{corollary}{Corollary} 

\newtheorem{definition}{Definition}
\newtheorem{Problem}{Problem}
\usepackage{tikz-cd} 
\newcommand{\REAL}{\ensuremath{\mathbb{R}}}

\newcommand{\beq}{\begin{equation}}
\newcommand{\eeq}{\end{equation}}
\newcommand{\bea}{\begin{eqnarray}}
\newcommand{\eea}{\end{eqnarray}}
\newcommand{\bean}{\begin{eqnarray*}}
\newcommand{\eean}{\end{eqnarray*}}
\newcommand{\bcen}{\begin{center}}
\newcommand{\ecen}{\end{center}}
\newcommand{\bitm}{\begin{itemize}}
\newcommand{\eitm}{\end{itemize}}

\newtheorem{exampl}{\bf Example}
\newcommand{\defineq}{\ensuremath{\, \triangleq \,}}    

\def\sgcnv{\hbox{$\, \bigcirc \,$\kern-0.9em\hbox{\mgop}$\,$}} 
\def\supgeno{\hbox{$\, \bigcirc \,$\kern-1.0em\hbox{$\wedge$}$\,$}} 
\def\infgeno{\hbox{$\, \bigcirc \,$\kern-1.0em\hbox{$\vee$}$\,$}} 
\newcommand{\mgop}{\ensuremath{\star}}  

\newcommand{\vct}[1]{\ensuremath{\boldsymbol{#1}}}  
\newcommand{\mtr}[1]{\ensuremath{\boldsymbol{#1}}}  
\newcommand{\mxsmp}{\ensuremath{\boxplus}} 
\newcommand{\mnsmp}{\ensuremath{\mxsmp'}} 

\newcommand{\abs}[1]{\ensuremath{\left\vert #1\right\vert}}
\newcommand{\norm}[1]{\ensuremath{\left\| #1\right\|}}
\newcommand{\paren}[1]{\ensuremath{\left( #1\right)}}
\newcommand{\clint}[1]{\ensuremath{\left[ #1\right]}}

\newcommand{\set}[1]{\ensuremath{\left\{ #1\right\}}}

\newcommand{\ssup}{\ensuremath{\vee}}
\newcommand{\sinf}{\ensuremath{\wedge}}

\newcommand{\mmult}{\ensuremath{\boxplus}}

\newcommand{\bigsup}[2]{\ensuremath{\bigvee\limits_{#1}^{#2}}}
\newcommand{\biginf}[2]{\ensuremath{\bigwedge\limits_{#1}^{#2}}}
\newcommand{\Rmax}{\ensuremath{\mathbb{R}_{\mathrm{max}}}}
\newcommand{\Rmin}{\ensuremath{\mathbb{R}_{\mathrm{min}}}}

\newcommand{\vect}[1]{\ensuremath{\clint{\begin{array}{r} #1 \end{array}}}}

\newcommand{\matr}[1]{\ensuremath{\clint{\begin{array} #1 \end{array}}}}

\newcommand{\solu}[2]{\ensuremath{S\paren{\mtr #1,\vct #2}}}

\newcommand{\x}{\ensuremath{\tilde{x}}}
\newcommand{\1}{\ensuremath{\mathbbm{1}}}

\DeclareMathOperator*{\opsup}{supp}
\newcommand{\supp}[1]{\ensuremath{\opsup{\paren{\vct{#1}}}}}
\newcommand{\argmin}{\arg\!\min}

 \hypersetup{
    citecolor=black,
    filecolor=black,
    linkcolor=green,
    urlcolor=black,
    linkbordercolor={0 1 0},
    citebordercolor={1 0 0}
}

\begin{document}
	\title{Sparsity in Max-Plus Algebra and Systems
		~\thanks{The paper was published in the  Discrete Event Dynamic Systems journal; doi:~\url{https://doi.org/10.1007/s10626-019-00281-1}. }
}
\author{Anastasios Tsiamis~\thanks{Department of Electrical and Systems Engineering, University of
		Pennsylvania, 200 South 33rd Street, Philadelphia, PA 19104, United States. (email: atsiamis@seas.upenn.edu)} and Petros Maragos~\thanks{School of Electrical and Computer Engineering, National Technical University of Athens, Zografou Campus, 15773 Athens, Greece. (email: maragos@cs.ntua.gr)}}
\maketitle

	\begin{abstract}		
	We study sparsity in the max-plus algebraic setting. We seek both exact and approximate solutions of the max-plus linear equation with minimum cardinality of support. In the former case, the sparsest solution problem is shown to be equivalent to the minimum set cover problem and, thus, NP-complete. In the latter one, the approximation is quantified by the $\ell_{1}$ residual error norm, which is shown to have supermodular properties under some convex constraints, called lateness constraints. Thus, greedy approximation algorithms of polynomial complexity can be employed for both problems with guaranteed bounds of approximation. 	
	We also study the sparse recovery problem and present conditions, under which, the sparsest exact solution solves it. Through multi-machine interactive processes, we describe how the present framework could be applied to two practical discrete event systems problems: resource optimization and structure-seeking system identification. We also show how sparsity is related to the pruning problem. Finally, we present a numerical example of the structure-seeking system identification problem and we study the performance of the greedy algorithm via simulations.
	\end{abstract}

\section{Introduction}\label{Section_Introduction}
Max-plus algebra has been used to model a subclass of nonlinear phenomena with some linear-like structure. It is obtained from the linear algebra if we replace addition with maximum and multiplication with addition \citep{Butk10}. The development of this algebraic theory was motivated by problems arising in scheduling theory, graph theory and operations research \citep{Cuni79}. Later on, max-plus algebra was also employed in discrete event systems to deal mainly with synchronization problems \citep{CDQV85,BCOQ92,1999Cohen,DeSchutter2008max}. Other applications include the max-algebraic approach to optimal control \citep{2001Litvinov,Mceneaney2006max}, general max-plus dynamical systems and control \citep{Adzkiya2015computational,2011Hardouin} and generalized HMMs for audiovisual event detection \citep{Maragos2015max}.
An extensive survey about the applications of the max-plus algebra can be found in~\cite{Gaubert2009max}.
Generalizations of max-plus algebra using other idempotent semirings are described in~\cite{gondran2008graphs}. A unification of max-type algebras and their duals using weighted lattices with applications to nonlinear dynamical systems was presented in~\cite{maragos2017dynamical}. 

Meanwhile, in the last decade, we have experienced an increase of interest in sparsity in linear equations and linear  systems. A solution of a linear equation is sparse when it has many zero elements. The reason we are interested in such solutions, is that they need less elements to describe the same information. They provide us a way of compressing the available data, throwing away those that are unnecessary \citep{Donoho2006compressed}. They also reveal the structure of partially known signals \citep{Candes2006robust} or systems \citep{ChenHero2009sparse}. In control systems, sparsity has been sought in the sense of minimizing the number of sensors or actuators, subject to energy \citep{2016Tzoumas,Summers2016submodularity} or observability-controllability constraints~\citep{Pequito2016minimum}.

 Although sparsity has been extensively studied in the linear setting \citep{elad2010sparse}, it is still not much developed in more general nonlinear settings. In this work, we aim to define and study sparsity in the max-plus algebraic setting. A sparse solution of a max-plus equation is a solution with many non-informative elements, i.e. the infinite elements. As in the linear case, such solutions use the least number of elements to describe the same information, thus yielding compressed data. But there are many other applications where sparsity could be relevant. For example, in max-plus systems (either static or dynamical), finding sparse inputs implies that we are activating fewer actuators/machines, thus, saving resources.
 Similarly, the problem of selecting few sensors to observe a max-plus system could be expressed in terms of designing sparse output matrices.
 Another application could be in max-plus system identification problems, where the sparsity structure is unknown. In this case, sparse solutions could be employed to reveal the unknown structure of the original system.

Our theoretical contributions are the following:
\begin{enumerate}[i)]
\item We define sparsity in the max-plus algebraic setting (see Section~\ref{Section_Problem}); a vector is defined to be sparse when it has many $-\infty$ elements. 
\item  We define the problem of finding the sparsest exact solution to the  max-plus equation (see problem~\eqref{problem_1}). Then, in Section~\ref{Section_Exact_Solution}, we show that this problem is equivalent to the minimum set-cover one and, thus, NP-complete (Theorem~\ref{basic_theorem_1}).
\item  We define the problem of finding the sparsest approximate solution  (see problem~\eqref{problem_2}). Here, we are searching for the sparsest solution that satisfies the following constraints: i) its $\ell_1$ approximation error is bounded and ii) it satisfies some additional convex constraints, called lateness constraints. In Section~\ref{section_approximate}, we show that the $\ell_1$-error of approximation has supermodular properties (Theorem~\ref{Supermodular_error_function}). Thus, a suboptimal greedy approximate algorithm of polynomial complexity can be employed with guaranteed bounds on the suboptimality ratio (Theorem~\ref{Theorem_supermodular_bounds_finite}). Our analysis is extended to the case when the components of the matrices are allowed to take $-\infty$ values (Theorem~\ref{Theorem_infinite_approximation_bounds}). 
\item We study the sparse recovery problem (see Section~\ref{Section_Inversion}, Theorem~\ref{Theorem_inversion}). In particular, it is explored whether we can recover a vector, for which we do not know the sparsity pattern, from its image under a max-plus linear transformation. We derive sufficient conditions, under which, we can use the sparsity framework to recover that vector and its sparsity pattern.
\end{enumerate}

The paper is organized as follows. In Section~\ref{section_max_algebra}, we revisit the max-plus equation and its properties. Section~\ref{Section_Problem} formulates the problems of finding the exact and approximate sparsest solutions to the max-plus equation. Then, in Sections~\ref{Section_Exact_Solution},~\ref{section_approximate}, we present possible solutions to the the former and the latter problem respectively. For completeness, in Section~\ref{section_approximate}, we also include a brief introduction to the supermodularity literature. In Section~\ref{Section_Inversion}, we study the sparse recovery problem. In Section~\ref{Section_Applications}, we study two applications of the sparsity framework to multi-machine interactive production processes~\citep{Butk10}: i) application to resource optimization and ii) application to structure-seeking system identification. There, we also show how our sparsity framework is related to the pruning problem~\citep{mceneaney2009pruning,gaubert2011curse}. In Section~\ref{Section_Example}, we present a numerical example of the system identification problem and we study the performance of the greedy algorithm via simulations. Finally, in Section~\ref{Section_Conclusion}, we conclude the paper and discuss possible extensions of the present work.
All proofs which do not appear in the main text are included in the Appendix.
\subsection{Related Work}
The relation between set covers and solutions to the max-plus equation has been known before~(\citep{vorobyev1967extremal,zimmermann1976extremal,Butk03,akian2005set}). 
We use those previous results to prove the equivalence between the sparsest exact solution problem and the minimum set cover problem in Theorem~\ref{basic_theorem_1}. 
Still, our paper is the first to explicitly define and study the problem of finding the sparsest exact solution.

The most related problem to sparsity is the pruning one~\citep{mceneaney2009pruning,gaubert2011curse}.
 It arises in optimal control problems, where we try to approximate value functions as the supremum of certain basis functions. The goal there is to replace the supremum over many basis functions with the supremum over a smaller subset of basis functions, which has fixed-cardinality; this subset is selected via minimizing an $\ell_1$ approximation error cost. The problem of finding the sparsest approximate solution (problem~\eqref{problem_2}) defined in this paper is a ``dual" version of the pruning one--see also Section~\ref{Section_Applications}. The minimization is over the cardinality of the subset such that the $\ell_1$-error remains bounded. Another difference is that the pruning problem deals with basis functions defined on infinite spaces; the sparsity problem deals with basis vectors instead of basis functions.
 In~\cite{gaubert2011curse}, equation~(34), the pruning problem is reduced to a $k-$median problem, which can be shown to have supermodular properties~\citep{1978Nemhauser}. This argument could lead to an alternative proof of Theorem~\ref{Supermodular_error_function}.
Finally, our sparsity framework also applies when the basis vectors are allowed to have $-\infty$ (null) components~(Theorem~\ref{Theorem_infinite_approximation_bounds}).

The recovery problem, without any sparsity considerations, is related to the uniqueness of the max-plus equation (see Chapter~15 of \cite{Cuni79} or~\cite{Butk10} or Corollary~4.8 in~\cite{akian2005set}). However, the sparse recovery problem is quite different; we might be able to solve the sparse recovery problem even if we have have infinite solutions to the max-plus equation--see Section~\ref{Section_Inversion} for more details. The results of~\cite{2006SchullerusSchutter} are related to the result of Theorem~\ref{Theorem_inversion}. However, in~\cite{2006SchullerusSchutter} there are no sparsity considerations, e.g. the sparsity pattern of the involved matrices is considered known. To the best of our knowledge, our paper is the first to define and address the sparse recovery problem.

\section{Notation and Background}\label{section_max_algebra}
Throughout this paper, matrices and vectors will be denoted by bold characters.  If $\mtr A$ is a $m\times n$ matrix then its columns ($m\times 1$ vectors) are denoted by $\vct{A}_{j}$, $j=1,\dots,n$. Its components are denoted by $A_{ij}$ or $\clint{\mtr A}_{ij}$, for $i=1,\dots,m$, $j=1,\dots,n$. The transpose matrix is denoted by ${\mtr A}^\intercal$. If $\vct x$ is a $n\times 1$ vector, its components are denoted by $x_j$ or $\clint{\vct x}_{j}$, for $j=1,\dots,n$. 
Finally, for simplicity we denote the row and column index sets by $I=\set{1,2,\dots, m}$ and $J=\set{1,2,\dots, n}$, respectively.
\subsection{The Max-Plus Algebra}
 The max-plus algebra (or the $\paren{\ssup,+}$ semiring) is the set $\Rmax=\REAL\cup \set{-\infty}$ equipped with the \emph{maximum} operator $\ssup$ as ``addition" and $+$ as ``multiplication" \citep{1997Gaubert}.\footnote{An alternative notation that has been used in the literature is $\oplus$ for maximum (max-plus ``addition") and $\otimes$ for addition (max-plus ``multiplication")--see~\cite{Cuni79} or~\cite{BCOQ92}. Here, we follow the notation of lattice theory--see~\cite{birkhoff1940lattice},~\cite{Mara13},~\cite{maragos2017dynamical}, where the symbol  $\vee/\wedge$ is used for max/min operations. 
 	 We also use the classic symbol "+" for real addition, without obscuring  the addition with the less intuitive symbol $\otimes$. Further, we avoid the symbol $\oplus$ because it is used in signal and image processing to denote max-plus convolution and in set theory to denote Minkowski set addition.  } If $x,y\in \Rmax$, then $x\ssup y\defineq \max\set{x,y}$. 
The zero  element for the maximum operator $\ssup$ is $-\infty$. The operator $+$ is defined in the usual way  with $0$ as the identity element and $-\infty$ as the null element. Similarly, we can define the min-plus algebra on $\REAL_{min}=\REAL\cup\set{+
\infty}$, equipped with the minimum operator $\sinf$ and addition $+$.

If $\vct x,\vct y \in \Rmax ^{n}$ are vectors, we overload $\ssup$ with componentwise maximum \[\clint{\vct x\ssup \vct y}_{i}=x_{i}\ssup y_{i},\,i=1,\cdots,n.\]
 Operators $<,\leq$ are interpreted with the vector partial order, induced by componentwise comparison. We also define the addition $\vct x +a$ of a scalar $a\in \Rmax$ to a vector $\vct x\in \Rmax ^{n}$ componentwise as follows:
\[\clint{\vct x+a}_i=x_i+a,\, i=1,\cdots,n,\]
This can be interpreted as the scalar ``multiplication" counterpart of linear algebra.

If $\mtr A\in \Rmax^{m\times n}$, $\mtr B\in \Rmax^{m\times n}$ are matrices, then we define their componentwise maximum $\clint{\mtr A\ssup \mtr B}_{ij}=A_{ij}\ssup B_{ij},\: i=1,\cdots,m,\:j=1,\cdots,n$. 
If $\mtr A\in \Rmax^{m\times n}$ and $\mtr B\in \Rmax^{n\times p}$, then 
 their max-plus matrix ``multiplication" $\mtr A\mxsmp \mtr B \in \Rmax^{m\times p}$ is defined as:
\begin{equation*}
\clint{\mtr A\mxsmp \mtr B}_{ij}=\bigsup{k=1}{n}\paren{ A_{ik}+ B_{kj}},\: i=1,\cdots,m,\:j=1,\cdots,p.
\end{equation*}
If $\mtr A\in \Rmin^{m\times n}$, $\mtr B\in \Rmin^{n\times p}$, their min-plus matrix multiplication $\mtr A \mnsmp \mtr B\in \Rmin^{m\times p}$ is defined similarly:
\begin{equation*}
\clint{\mtr A\mnsmp \mtr B}_{ij}=\biginf{k=1}{n}\paren{ A_{ik}+ B_{kj}},\: i=1,\cdots,m,\:j=1,\cdots,p.
\end{equation*}
\subsection{Max-plus Linear Equation and Exact Solution}
The max-plus equation has a form similar to the linear equation $\mtr A \vct x=\vct b$, though we replace addition with maximum and multiplication with addition. 
In particular, given  $\mtr A\in\Rmax^{m\times n}$, $\vct x\in \Rmax ^{n}$, $\vct b\in \REAL^{m}$, it is given by the following formula\footnote{In the general case, $\vct b\in \Rmax^{m}$ \citep{Butk10}. However, in this paper we will only consider finite $\vct b\in \REAL^{m}$. See also Assumption~\ref{assuption_A} in Section~\ref{Section_Problem}.}:
\begin{equation*}
\bigsup{j=1}{n}\paren{A_{ij}+x_{j}}=b_{i},\: i=1,\cdots,m
\end{equation*}
or in compact form
\begin{equation} \label{definition_1}
\mtr A\mxsmp \vct x=\vct b.
\end{equation}
Next,	we define the \emph{set of all solutions} of \eqref{definition_1} $S\paren{\mtr A, \vct b}$:
	\begin{equation}\label{EQN_solution_set}
S\paren{\mtr A, \vct b}=\set{\vct x \in \Rmax^n:\mtr A\mxsmp \vct x=\vct b}
	\end{equation}
We can also write \eqref{definition_1} as $\mtr A\mxsmp \vct x=\bigsup{j=1}{n}\paren{ \vct A_{j}+x_{j}}=\vct b$. 
Hence, in this form,  $\mtr A\mxsmp \vct x$ can be interpreted as a ``max-plus linear combination" of the columns of $\mtr A$ with weights $x_j$. 

 To analyze the max-plus equation, we need the definition of the \emph{principal solution} $\vct{\bar{x}}\in \Rmin^{n}$ \citep{Cuni79} \footnote{The principal solution can also be expressed in terms of residuation theory--see, for example,~\cite{BCOQ92}. The map $\Pi(\vct x)=\mtr A\mxsmp \vct x$ is residuated, with $\Pi^{\sharp}(\vct b)=(-\mtr A^{\intercal})\mnsmp \vct b$ being the residual map, where $\mnsmp$ denotes the min-plus matrix product. Both maps are increasing and they satisfy the property $\paren{\Pi\circ\Pi^{\sharp}}(\vct b)\le \vct b$, $\paren{\Pi^{\sharp}\circ\Pi}(\vct x)\ge \vct x$. Then, the principal solution $\vct{\bar{x}}$ can be written as $\vct{\bar{x}}=\Pi^{\sharp}(\vct b)$. The notion of residuated and residual maps is also related to the notion of adjunctions in lattice theory, e.g. see~\cite{Mara13},~\cite{maragos2017dynamical}, as well as the notion of Galois Connections, e.g see~\cite{akian2005set}.}:
 \[
 \vct{\bar{x}}=\paren{-\mtr A}^{\intercal}\mnsmp \vct b,
 \]
 whose components can be expressed as:
\begin{equation}\label{principal_solution}
\bar{x}_{j}=\biginf{i=1}{m}\paren{b_{i}-A_{ij}},\,\forall j\in J.
\end{equation}
Although the principal solution belongs to $\Rmin^n$, in this paper we will only deal with cases where $\vct{\bar{x}}\in \REAL^{n}$. When the max-plus equation~\eqref{definition_1} admits a solution, it turns out that the principal solution $\bar{\vct{x}}$ is also an actual solution (see Theorem~\ref{theorem_multiple_solutions} in the Appendix). In other words,  the set $S\paren{\mtr A, \vct b}$ is non-empty if and only if $\vct{\bar{x}}$ is a solution to equation \eqref{definition_1} \citep{Cuni79}.
\subsection{Max-plus Linear Equation and Approximate Solution}
Although the principal solution $\bar{\vct{x}}$ is always defined, it may not be a solution of \eqref{definition_1}. In this case, system~\eqref{definition_1} cannot be solved. However, we may find an approximate solution, by minimizing the $\ell_{1}$ norm of the residual error $\vct b - \mtr A \mxsmp \vct x$. Still, without any additional constraint this problem is hard to solve.
For this reason, the convex constraint 
\begin{equation}
\mtr A \mxsmp \vct x\le \vct b, \label{lateness_constraint}
\end{equation}
also called the \emph{lateness constraint} \citep{Cuni79}, is added to the minimization problem. This relaxation, adopted in \cite{Cuni79}, is also motivated by time constraints in operations research (see also Section~\ref{Section_Applications_Resource}). The approximate solution problem can be described with the following optimization problem:
\begin{equation}\label{EQN_approximate}
\begin{aligned}
&\underset{x \in\Rmax^n}{\textup{minimize}} &&\norm{\vct b - \mtr A \mxsmp \vct x}_{1} \\
&\textup{subject to} && \mtr A \mxsmp \vct x\le \vct b,
\end{aligned} 
\end{equation}
which can be recast as a linear program. 
It turns out that the principal solution $\vct{\bar{x}}$ is the largest possible element that satisfies the constraint $\mtr A \mxsmp \vct x\le \vct b$. Therefore, it is also an optimal solution to problem~\eqref{EQN_approximate} (see Theorem~\ref{Theorem_approximate_max_plus_solution} in the Appendix).

 \section{Problem Statement}\label{Section_Problem}
 In this section, we define the problem of finding the sparsest exact and approximate solutions to the max-plus equation $\mtr A\mxsmp \vct x=\vct b$. In linear algebra, the sparsity pattern of a vector or a matrix is determined by the set of its nonzero components.
 In a similar fashion, in max-plus algebra,  the sparsity pattern of any matrix or vector is determined by the set of its finite elements, since the zero element is $-\infty$.  
We define the \emph{support} of an element $\vct x\in \Rmax^{n}$ as 
\[\supp{x}=\set{j\in J:\:x_{j}\neq -\infty},\]
i.e. the set of the indices of its finite components. 

The first problem studied in this paper is finding the sparsest solution to  equation~\eqref{definition_1}.
 Formally, given the matrices $\mtr A\in \Rmax^{m\times n},\,\vct b\in \REAL^{m}$, we want to determine the optimal (possibly non-unique) solution $\vct x^{*}\in\Rmax^{n}$ to the following optimization problem:
 \begin{equation}\label{problem_1}
 \begin{aligned} 
 \vct x^{*}=&\argmin\limits_{\vct x\in\Rmax^{n}} &&\abs{\supp{x}}\\
 &\text{subject to}&&\mtr A\mxsmp \vct x=\vct b
 \end{aligned}
 \end{equation}
 where $\abs{T}$ denotes the cardinality of a set $T$.
 
However, a solution to equation $\mtr A\mxsmp \vct x=\vct b$ may not exist. Meanwhile, solving problem~\eqref{EQN_approximate} might not work either, since it does not guarantee a sparse approximate solution. 
Instead of optimizing with respect to the residual error, one option would be to search for sparse approximate solutions to~\eqref{definition_1}, within some allowed error. 
We define an $\epsilon$\emph{-approximate solution} to \eqref{definition_1} as a vector $\vct x\in \Rmax^{n}$ that: i) has residual error bounded by positive constant $\epsilon>0$ or $\norm{\vct b-\mtr A\mxsmp \vct x}_{1}\le \epsilon$, and ii) satisfies the lateness constraint $\mtr A\mxsmp \vct x\le \vct b$.
 
 In the second problem, given a prescribed constant $\epsilon>0$, we seek the sparsest (possibly non-unique) $\epsilon$-approximate solution. Equivalently, we solve the optimization problem:
 \begin{equation}\label{problem_2}
 \begin{aligned} 
 \vct x^{*}=&\argmin_{x\in\Rmax^{n}}&& \abs{\supp{x}}\\
&\text{subject to}&&\norm{\vct b-\mtr A\mxsmp \vct x}_{1}\le \epsilon\\
& &&\mtr A\mxsmp \vct x\le \vct b
 \end{aligned}
 \end{equation}
 We may recover the exact sparsest solution problem if we select $\epsilon=0$. Notice that we need  to select $\epsilon\ge \norm{\vct b-\mtr A\mxsmp \vct{\bar x}}_{1}$ in order to guarantee feasibility of problem~\eqref{problem_2} (follows from~Theorem~\ref{Theorem_approximate_max_plus_solution}).

 To guarantee that the problem we are solving is not trivial, we make the following assumption about  $\mtr A$ and  $\vct b$, which \emph{holds throughout the paper}. It has been a standard assumption in the literature (see chapter 15 in \cite{Cuni79}).
 \begin{assumption}\label{assuption_A}
 	All elements of $\vct b$ in equation \eqref{definition_1} are finite: $\vct b\in \REAL^{m}$. Every row and column of matrix $\mtr A\in\Rmax^{m\times n}$ in \eqref{definition_1} has at least one finite element\footnote{Such matrices are also called doubly $\REAL$-astic in \cite{Butk10} and doubly G-astic in \cite{Cuni79}.} :
 	\begin{align*}
 	&\text{$i$-th row: }&&\bigsup{k=1}{n}A_{ik}\neq -\infty,\quad i=1,\dots,m\\
 	&\text{$j$-th column: }&&\bigsup {l=1}{m}A_{lj}\neq -\infty,\quad j=1,\dots,n
 	\end{align*}
 \end{assumption} 
 If this assumption is not satisfied, it leads to trivial situations. For example, if the $k$-th column $\vct A_{k}$ consists only of $-\infty$ elements, then $x_{k}$ does not influence the solution at all, since $A_{ik}+x_{k}=-\infty,\: i=1,\dots m$ for every $\vct x \in \Rmax^{n}$. So, we may remove $k$-th column  and variable $x_{k}$ without any effect.
 
 \begin{remark}
 The lateness constraint $\mtr A\mxsmp \vct x\le \vct b$ is desirable in many discrete-event systems applications (see also Section~\ref{Section_Applications_Resource}), where we want some tasks to be completed at times $\mtr A\mxsmp \vct x$, before some deadlines $\vct b$. In general, it makes problem~\eqref{problem_2} more tractable. However, in other situations where it is not needed, it might lead to less sparse solutions or higher residual error. It is a subject of future work to explore how we could remove it in problem~\eqref{problem_2}.
 \end{remark}

 \begin{remark}
 	The sparsest solution problem makes sense even if $m>n$ and we have an overdetermined system. When system \eqref{definition_1} is solvable, we might have infinite solutions. Among those solutions some might be sparse. 
 \end{remark}

In the following sections we study problems~\eqref{problem_1},~\eqref{problem_2}. Then, we explore the sparse recovery problem as well as applications.

\section{Sparsest Exact Solution}\label{Section_Exact_Solution}
In this section, we present our results about the solution to the first problem~\eqref{problem_1}. 
We show that the sparsest solution problem is equivalent to a minimum set cover problem and, thus, NP-complete. Recall that $J=\set{1,\dots,n}$ is used for column indices, while $I=\set{1,\dots,m}$ is used for row indices.

Although the principal solution $\vct{\bar{x}}$ defined in \eqref{principal_solution} is a solution when $S(\mtr A,\,\vct b)$ is non-empty, it is not sparse, as the next result shows. 
\begin{lemma}\label{LEMMA_finite_principal}
	Under Assumption \ref{assuption_A}, the principal solution $\vct{\bar{x}}$ of \eqref{definition_1}, defined in \eqref{principal_solution}, is finite or equivalently $\vct{\bar{x}}\in \REAL^{n}$. \hfill $\diamond$ 
\end{lemma}
\begin{proof}
Since $b_i$ is finite (Assumption~\ref{assuption_A}), for every $i\in I,\,j\in J$, we have $b_i-A_{ij}>-\infty$. Thus, $\bar{x}_{j}>-\infty$, for all $j\in J$. Moreover, from Assumption~\ref{assuption_A}, for every $j\in J$, there exists at least one $k\in I$, such that $A_{kj}$ is finite, which implies $b_k-A_{kj}$ is finite. Thus, \[\bar{x}_j=\biginf{i=1}{m}\paren{b_{i}-A_{ij}}\le b_k-A_{kj}<+\infty,\] 
and $\bar{x}_j$ is finite for all $j\in J$. \hfill $\qed$
\end{proof}

The above result implies that we should find another way to compute sparse solutions. In particular,  we can leverage results from \cite{vorobyev1967extremal}, \cite{zimmermann1976extremal}, (see \cite{Butk03} for an English source or Theorem~\ref{theorem_multiple_solutions} in the Appendix) which show that any solution of equation~\eqref{definition_1} has to agree with the principal one at some components. To each element $\vct x \in S\paren{\mtr A,\vct b}$, we assign the set of indices $J_{x}$, which indicates the components where $\vct x$ agrees with $\vct{\bar{x}}$:
\begin{equation}\label{EQN_agreement_set}
J_{x}=\set{j\in J: x_{j}=\bar{x}_{j}}
\end{equation}
We will call this set the \emph{agreement set} of $\vct x $. By Lemma~\ref{LEMMA_finite_principal}, since $\vct{\bar x}$ if finite, we have $J_x\subseteq \supp{x}$ for every solution $\vct x\in\solu{A}{b}$. 
The main idea is that if $\vct x\in\solu{A}{b}$ is a solution, we can construct a new sparser solution $\vct{\hat{x}}\in\solu{A}{b}$ such that $\supp{\hat{x}}=J_x$. Thus, solving the sparsest solution problem is equivalent to finding an agreement set of the smallest possible cardinality $\abs{J_x}$.

However, we cannot have arbitrarily small agreement set $J_x$. There are some necessary  conditions that should be satisfied. For each $j\in J$, we define the set of row indices $I_j\subseteq I$, where the minimum in \eqref{principal_solution} is attained: 
\begin{equation} \label{sets_which_cover}
I_{j}=\set{i\in I:\quad b_{i}-A_{ij}=\biginf{k=1}{m}\paren{b_{k}-A_{kj}}=\bar{x}_{j}}.
\end{equation}
 Those necessary conditions require the sub-collection $I_j$, $j\in J_x$ to be a set cover of $I$ (see Theorem~\ref{theorem_multiple_solutions} in the Appendix).

The next theorem proves that the solution to problem \eqref{problem_1} can be reduced to finding the minimum set cover of $I$, by the subsets $I_j$, $j\in J$; the minimum is with respect to the number of subsets required for the cover. Conversely, any minimum set cover problem can be reduced to solving an instance of problem \eqref{problem_1}, for suitably defined matrices $\mtr A,$ $\vct b$. Thus, problem \eqref{problem_1} is NP-complete.

\begin{theorem} \label{basic_theorem_1}
	
\noindent	i) The problem \eqref{problem_1}  of computing the sparsest max-plus solution  is equivalent to finding the minimum set cover of $I$ by the subset-collection $\set{I_{j}:\:j\in J}$ defined in \eqref{sets_which_cover}.
	In particular, let $\vct{\bar{x}}$ be the principal solution defined in~\eqref{principal_solution}. Given a minimum set cover $\set{I_{j}:\:j\in K^{\star}}$, $K^{\star}\subseteq J$, the element $\vct{\hat{ x}}\in \Rmax^{n}$ defined as:
	\begin{equation}\label{EQN_sparse_solution}
	\begin{aligned}
	\hat{x}_{j}&=\bar{x}_{j},&&\: j\in K^{\star}\\
	\hat{x}_{j}&=-\infty,&&\: j\in J\setminus K^{\star},
	\end{aligned}
	\end{equation}
	is an optimal solution to problem~\eqref{problem_1}.

\noindent	ii)		 Any minimum set cover problem can be reduced to solving problem \eqref{problem_1}, for suitably defined matrices $\mtr A,$ $\vct b$. Thus, problem \eqref{problem_1} is NP-complete.
	 \hfill $\diamond$ 
\end{theorem}

\begin{remark}[Suboptimal solution to problem~\eqref{problem_1}]
	According to Theorem~\ref{basic_theorem_1}, we can solve problem~\eqref{problem_1}, by finding the minimum set cover $\set{I_j:\:j\in K^{*}}$ of $I$, and by using~\eqref{EQN_sparse_solution} to construct an optimal solution $\vct{ x^*}$. Although the minimum set cover is an NP-complete problem, it can be approximated by a greedy algorithm of polynomial complexity with approximation ratio $1+\log\paren{n}$ \citep{chvatal1979}. 
	Alternatively, we could solve problem~\eqref{problem_1} by solving problem~\eqref{problem_2} for $\epsilon=0$, using the techniques of Section~\ref{section_approximate}.
\end{remark}
The next example illustrates the results of this section.
\begin{exampl} \label{EXAMPLE_sparse}
	\normalfont 
	Suppose we are given the equation
	\begin{equation*}
	\matr{{rrr} 1&0&1 \\ -2&2&1\\ 1&1&0}\mmult\vect{x_{1}\\x_{2}\\x_{3}}=\vect{2\\0\\2}
	\end{equation*}
	From \eqref{principal_solution}, the principal solution is 
	\[\vct{\bar x}=\matr{{c}(2-1)\sinf (0+2)\sinf(2-1)\\(2-0)\sinf (0-2)\sinf (2-1)\\ (2-1)\sinf (0-1) \sinf (2-0)}=\matr{{r}1\\-2\\-1}.\]
	From \eqref{sets_which_cover}, the row index sets $I_j$ are:
	\[
	I_1=\set{1,3},\,I_2=\set{2},\,I_3=\set{2}.
	\]
	The minimum set cover of $I=\set{1,2,3}$ is either $I_{1}\cup I_{2}$ or $I_{1}\cup I_{3}$. Hence, we have two possible sparsest solutions: $\vct x^{*}_{1}=\matr{{ccc}1&-2&-\infty}^{T}$ and $\vct x^{*}_{2}=\matr{{ccc}1&-\infty&-1}^{T}$.\hfill $\diamond$
\end{exampl}

\section{Approximate Solution and Supermodular Approach}\label{section_approximate}
In this section, we present the approximate solution to problem~\eqref{problem_2}, which uses tools from the supermodular optimization literature; a brief introduction to supermodularity is included in Subsection~\ref{Section: Supermodular}. In Subsection~\ref{Section: Reformulation}, we reformulate problem~\eqref{problem_2} to a simpler one, where we only optimize over the support of the optimal solution. Then, in Subsection~\ref{Section: Finite Case}, we prove that this new optimization problem has supermodular properties if $\mtr A$ has only finite elements (Theorem~\ref{Supermodular_error_function}). This allow us to approximately solve problem~\eqref{problem_2} via a greedy algorithm of polynomial complexity with guaranteed bounds of approximation (Theorem~\ref{Theorem_supermodular_bounds_finite}). In some sense, this greedy solution is similar to the ``matching pursuit" algorithm in \cite{MallatZhang1993}, applied to linear systems.
Finally, in Subsection~\ref{Section: Infinite Case}, we extend the results to the case where matrix $\mtr A$ can also have infinite elements (Theorem~\ref{Theorem_infinite_approximation_bounds}).   

\subsection{Supermodularity Preliminaries}\label{Section: Supermodular}
Supermodularity \citep{2012Krause} is a property of set functions, which enables us to approximately solve some optimization problems of combinatorial complexity. In particular, greedy algorithms of polynomial complexity can be employed, with theoretical guarantees (bounds) regarding the ratio of approximation \citep{1982Wolsey}, \citep{1978Nemhauser}.
A \emph{set function} $f:2^{J}\rightarrow \REAL$
is a function that takes a subset $T\subseteq J$ and returns a real value $f\paren{T}$. 
Two useful properties of set functions are supermodularity and monotonicity. A set function $f:2^{J}\rightarrow \REAL$ is \emph{supermodular} if for every $C\subseteq B\subseteq J$ and $k\in J$:
\begin{equation}\label{definition_supermodular}
f\paren{C\cup \set{k}}-f\paren{C}\le f\paren{B\cup \set{k}}-f\paren{B}
\end{equation}
Respectively, a set function $f:2^{J}\rightarrow \REAL$ is \emph{decreasing} if for every $C\subseteq B\subseteq J$, $f\paren{C}\ge f\paren{B}$. 

Finally we present a result from \cite{1982Wolsey}\footnote{The result in \cite{1982Wolsey} is for submodular and increasing functions. But $f$ is supermodular (decreasing) if and only if $-f$ is submodular (increasing). Hence, the result is also valid for supermodular and decreasing functions.}, which shows how we can approximately solve cardinality minimization problems subject to a supermodular equality constraint. Let the optimization problem be:
\begin{equation}\label{cardinality_problem}
\begin{aligned}
&\underset{T\subseteq J}{\text{minimize}}&& \abs{T}\\
&\text{subject to }&& f\paren{T}= f\paren{J}
\end{aligned}
\end{equation}
where $f:2^{J}\rightarrow \REAL$ is supermodular and decreasing, while $\abs{T}$ denotes the cardinality of set $T$.
Suppose we use the following greedy algorithm. 
\begin{algorithm}
	\caption{Greedy Approximate Solution of \eqref{cardinality_problem}}
	\label{Algorithm_preliminaries}
	\begin{algorithmic}[1] {}
		\State{Set $T_{0}=\emptyset$, $k=0$} 
		\While{$f\paren{T_k}\neq f\paren{J}$}
		\State{$k=k+1$}
		\State{$j=\argmin\limits_{s\in J\setminus T_{k-1}}\set{f\paren{T_{k-1}\cup \set{s}}}$}
		\State{$T_{k}=T_{k-1}\cup\set{j}$}
		\EndWhile  
		\State{\Return$T_{k}$}
	\end{algorithmic}
\end{algorithm}

The following theorem provides a bound on the approximation ratio of Algorithm~\ref{Algorithm_preliminaries}.
\begin{theorem}[\cite{1982Wolsey}] \label{Theorem_supermodular_bounds}
	Suppose $f:2^J\rightarrow \REAL$ is supermodular and decreasing. Algorithm \ref{Algorithm_preliminaries} returns a suboptimal solution $T_{k}\subseteq J$ to problem \eqref{cardinality_problem} with $\abs{T_{k}}=k$. If $T^{*}$ is the optimal solution then the following bound holds
	\begin{equation}\label{basic_bound}
	\frac{\abs{T_{k}}}{\abs{T^{*}}}\le 1+\log\paren{\frac{f\paren{\emptyset}-f\paren{J}}{f\paren{T_{k-1}}-f\paren{J}}}
	\end{equation}
	\hfill $\diamond$
\end{theorem}
In the next sections, we reformulate problem~\eqref{problem_2} in order to reveal its supermodular structure and leverage the results of Theorem~\ref{Theorem_supermodular_bounds}.

\subsection{Reformulation of Problem~\eqref{problem_2}}\label{Section: Reformulation}
 Given any feasible point \footnote{The feasible points of an optimization problem are the elements that satisfy the constraints.} $\vct x$ of problem~\eqref{problem_2}, we can construct a new  one by forcing every component in the support to be equal to the respective component of the principal solution. In this way, we reduce problem~\eqref{problem_2} to just finding the support of $\vct x$, skipping the decision over the finite values of $\vct x$. 
Formally, suppose $\vct x\in \Rmax^{n}$, with support $\supp{x}=T$, satisfies the inequality $\mtr A \mxsmp \vct x\le \vct b$. Now, define a new element $\vct z\in \Rmax^{n}$ with the same support as $\vct x$, $\supp{z}=T$. Then, replace all its finite components with the ones of the principal solution: $z_{j}=\bar{x}_{j},\: j\in \supp{x}$. In terms of the agreement set defined in \eqref{EQN_agreement_set}, we have $J_{z}=\supp{z}=T$.  The next lemma shows that the new vector $\vct z$ not only is feasible, but also has smaller residual error than $\vct x$. 

\begin{lemma} \label{Proposition_principal_fixed_support}
Fix a subset $T\subseteq J$. Let \[X_{T}=\set{\vct x\in \Rmax^{n}:\:\supp{x}=T,\,\mtr A\mxsmp\vct x\le \vct b}\] be the set of elements which satisfy the lateness constraint and have support equal to $T$. Assume that $\vct z\in \Rmax^{n}$ has support and agreement set equal to $T$:
\begin{align*}
J_z&=T\\
\supp{z}&=T.
\end{align*} Then, $\vct z\in X_{T}$ and  \[\norm{\vct b-\mtr A\mxsmp \vct x}_{1}\ge \norm{\vct b-\mtr A\mxsmp \vct z}_{1},\] for all $\vct x\in X_{T}$.  \hfill $\diamond$ 
\end{lemma}
 Since for any fixed support $\supp{x}=T\subseteq J$, we can select $x_{j}=\bar{x}_{j},\:j\in T$ and $x_{j}=-\infty ,\:j\in J\setminus T$, the only decision variable that matters in problem~\eqref{problem_2} is $T\subseteq J$. 
To introduce more compact notation, we can rewrite $\mtr A\mxsmp \vct x=\bigvee_{j\in J}\paren{ \vct A_{j}+x_{j}}$ as a max-plus linear combination.
But if $\supp{x}=T\subseteq J$, then this max-plus linear combination becomes:
\[
\mtr A\mxsmp \vct x=\bigvee_{j\in T}\paren{ \vct A_{j}+x_{j}},\text{ if }\supp{x}=T.
\]
Choosing $x_{j}=-\infty$ is equivalent to ignoring column $\vct A_{j}$ in the max-plus linear combination.
The next definition uses this notation.

\begin{definition}\label{DEF_error_functions}
	We define the  \textbf{error vector} $\vct e: 2^{J}\rightarrow \Rmin^{m}$ as:
	\begin{equation}\label{EQN_residual_error_vector}
	\begin{aligned}
	\vct e(T)&=\vct b-\bigsup{j\in T}{}\paren{\vct A_{j}+\bar{x}_{j}},\text{ for }T\neq \emptyset\\
	\vct e(\emptyset)&=\bigsup{j\in J}{}\vct e(\set{j}).
	\end{aligned}
	\end{equation}
	The \textbf{$\mathbf{\ell_{1}}$-error function} $E(T): 2^{J}\rightarrow \Rmin$ is defined as the $\ell_{1}$-norm of the error vector:
	\begin{equation}\label{EQN_residual_error_norm}
	E\paren{T}=\norm{\vct e(T)}_1,
	\end{equation}
	where $\norm{\vct e(T)}_1=\infty$ if $e_j(T)=\infty$, for some $j\in J$.
  \hfill $\diamond$
\end{definition}
We note that for the empty set we consider the singletons' error vectors and take the component-wise maximum in the above definition.  This selection guarantees that the $\ell_1$-error function $E$ is supermodular and decreasing. 

 The next corollary exploits the result of Lemma~\ref{Proposition_principal_fixed_support} and proves that we can rewrite problem~\eqref{problem_2} as:
\begin{equation}\label{EQN_reformulated_problem}
\begin{aligned}
 &\min \abs{T}\,\text{subject to } E\paren{T}\le \epsilon
\end{aligned}
\end{equation}
\begin{corollary}\label{COR_AlgorithmsEquivalent}
Problem~\eqref{problem_2} is equivalent to problem~\eqref{EQN_reformulated_problem}. In particular, if $\hat{T}$ is a optimal solution to problem~\eqref{EQN_reformulated_problem}, then the element $\vct{\hat{x}}\in \Rmin^{n}$ defined as:
\begin{equation}\label{EQN_reconstructed_solution_from_reformulation}
\begin{aligned}
\hat{x}_{j}&=\bar{x}_{j},\: j\in \hat{T}\\
\hat{x}_{j}&=-\infty,\: j\in J\setminus \hat{T}, 
\end{aligned}
\end{equation}
 is an optimal solution to problem~\eqref{problem_2}.\hfill $\diamond$
\end{corollary}

\subsection{Finite Element Case}\label{Section: Finite Case}
Now, we can show that if $\mtr A$ has only finite elements, the $\ell_1$ error set function $E\paren{T}$, defined in \eqref{EQN_residual_error_norm}, is supermodular. An alternative proof can be found if we follow the steps of~\cite{gaubert2011curse}, Section~VI\footnote{ Function $E(T)$ can be expressed as the cost function of a $k$-median problem--see~\cite{gaubert2011curse}. This function is known to be supermodular~\citep{1978Nemhauser}.}. 
\begin{theorem} \label{Supermodular_error_function}
Suppose $\mtr A\in\REAL^{m\times n}$.  The $\ell_{1}$ error set function $E(T)$ defined in \eqref{EQN_residual_error_norm}, is decreasing and supermodular. \hfill $\diamond$
\end{theorem}

The above result along with Corollary~\ref{COR_AlgorithmsEquivalent} enable us to approximately solve problem~\eqref{problem_2}, using  Algorithm~\ref{infinite_approximation_algorithm} below. 
First, we compute the approximate solution to problem~\eqref{EQN_reformulated_problem} in a greedy way. Define $T_k\subset J$ to be the collection of $k$ elements, selected greedily in a sequential way. Starting from the empty set $T_{0}=\emptyset$, at each time  $k$, we select the index $j$, which achieves the smallest $\ell_1$-error $E\paren{T_{k-1}\cup\set{j}}$.  Then, we update $T_k=T_{k-1}\cup\set{j}$ and this is repeated until the error $E\paren{T_{k}}$ becomes less than~$\epsilon$. After the selection of $T_k$, we construct an approximate solution according to equation~\eqref{EQN_reconstructed_solution_from_reformulation}. 
The  complexity of the algorithm is $\mathcal{O}(n^2)$, since the minimization step requires an inner loop of at most $n$ iterations, while the outer loop requires at most $n$ iterations.
\begin{algorithm} 
 \caption{Approximate Solution of Problem \eqref{problem_2}}
 \label{infinite_approximation_algorithm}
 \begin{algorithmic}[1]{}
  \Require $\mtr A,\: \vct b$
 \State{Compute $\vct{\bar{x}}$ from \eqref{principal_solution}}
   \If{$E(J)>\epsilon$ }\State{ \Return{Infeasible}}
 \EndIf
   \State{Initialize $\hat x_j=-\infty$, for all $j\in J$}
 \State{Set $T_{0}=\emptyset$, $k=0$} 
 \While{$E\paren{T_{k}}>\epsilon$}
	\State{$k=k+1$}
	\State{$j=\argmin\limits_{s\in J\setminus T_{k-1}}E\paren{T_{k-1}\cup \set{s}}$}
 	\State{$T_{k}=T_{k-1}\cup\set{j}$}
 \EndWhile  
 \State{Update $\hat{x}_j=\bar{x}_j,$ $j\in T_{k}$}
 \State{\Return $\vct{\hat{x}}$,\,$T_{k}$}
 \end{algorithmic}
\end{algorithm}

Since $E\paren{T}$ is a supermodular function, it follows that
$\bar{E}(T)=\max\paren{E\paren{T},\epsilon}$ is also supermodular~\citep{2012Krause}. Thus, the constraint $E\paren{T_{N}}>\epsilon$ is equivalent to $\bar{E}(T)=\epsilon$. Now, by applying the results of~\cite{1982Wolsey} (Theorem~\ref{Theorem_supermodular_bounds}), we can obtain an upper bound to the  approximation ratio of Algorithm~\ref*{infinite_approximation_algorithm}.

\begin{theorem} \label{Theorem_supermodular_bounds_finite}
Assume that $\mtr A\in \REAL^{m\times n}$ has only finite elements. Suppose $\epsilon\ge 0$ is such that $E\paren{J}\le \epsilon$ and $E\paren{\emptyset}>\epsilon$, where $E$ is defined in~\eqref{EQN_residual_error_norm}. Let $k$ be the time Algorithm~\ref{infinite_approximation_algorithm} terminates with $\vct {\hat{x}}$, $T_k$ the respective outputs.
Then,  $\vct {\hat{x}}$ is a suboptimal solution to problem~\eqref{problem_2} with $T_k=\supp{ {\hat{x}}}$. 
 Moreover, if $T^{*}=\supp{x^*}$, where $\vct x^*$ is an optimal solution of problem~\eqref{problem_2}, the following inequality holds:
\begin{equation}\label{basic_bound_approximation_error}
\frac{\abs{T_k}}{\abs{T^{*}}}\le 1+\log\paren{\frac{m\Delta}{E\paren{T_{k-1}}-\epsilon}}
\end{equation}
where $\Delta=\bigsup{i\in I,j\in J}{}\paren{b_{i}-A_{ij}-\bar{x}_{j}}$ and $\bar{x}_j$ are the components of the principal solution defined in~\eqref{principal_solution}.
 \hfill $\diamond$
\end{theorem}
Parameter $\Delta$ is the largest  element of the normalized matrix $[b_i-A_{ij}-\bar{x}_j]$, $i\in I,\,j\in J$.  Since $\mtr A$ has only finite elements, $\Delta$ is also finite.
The presence of the logarithm mitigates the effect of a large $\Delta$ or small $E\paren{T_{k-1}}-\epsilon$ differences.  
In general, term $E\paren{T_{k-1}}$ depends on $\mtr A,\,\vct b$,  but by allowing more memory, it can be precomputed for all possible $k$ with complexity $\mathcal{O}(n^2)$ ($\mathcal{O}(n)$ per $k$). 

  Nonetheless, there are special cases, where data independent bounds for the difference $E\paren{T_{k-1}}-\epsilon$ are possible. For example, if both $\mtr A$ and $\vct b$ are integer valued, which is common in timing applications, then the error function is also integer valued and $E\paren{T_{k-1}}\ge \lfloor\epsilon+1\rfloor$. Then, the bound of Theorem~\ref{Theorem_supermodular_bounds_finite} becomes $\frac{\abs{T_{k}}}{\abs{T^{*}}}\le 1+\log\paren{\frac{m\Delta}{\lfloor\epsilon+1\rfloor-\epsilon}}$ and does not depend any more on the specific $\mtr A,\,\vct b$. 
Quantized elements can also be dealt in a similar fashion.

\subsection{Infinite Element Case}\label{Section: Infinite Case}
If $\mtr A$ has infinite elements, then we cannot directly apply the results of Theorem~\ref{Theorem_supermodular_bounds}. However, we can replace the infinite elements of the error vector $\vct e\paren{T}$, $T\subseteq J$ with a sufficiently large positive constant $M>0$ and then exploit the results of the finite case.
The idea to replace infinite elements with big constants $M$ is motivated by the ``big-M" method in linear optimization~\citep{Bertsimas1997}.

It is sufficient
 to replace matrix $\mtr A\in \Rmax^{m\times n}$ with a new one, denoted by $\mtr{\hat{A}}(M)\in \Rmax^{m\times n}$, where:
\begin{equation}\label{EQN_modified_matrix}
\hat{A}_{ij}\paren{M}=\left.\begin{aligned}
&A_{ij},&&\text{ if }A_{ij}\neq -\infty\\ -&M+b_{i}-\bar{x}_{j},&&\text{ if }A_{ij}= -\infty
\end{aligned}\right\},\text{ for all }i\in I,\,j\in J.
\end{equation}
This new matrix $\hat{A}\paren{M}$ has only finite elements. Thus, we can now apply 
 Algorithm~\ref{infinite_approximation_algorithm} to matrices $\mtr {\hat{A}}\paren{M},\,\vct b$ instead of $\mtr {A},\:\vct b$ and leverage Theorem \ref{Theorem_supermodular_bounds_finite} to bound the approximation ratio. However, we first have to require that the optimal solution remains the same with this change. This is indeed the case if $M$ is large enough. In particular, if $M>\epsilon$, it turns out that the optimal solution remains the same, as the following lemma shows.
\begin{lemma} \label{LEM_New_Solution_Same}
Suppose $M>\epsilon\ge 0$.
Then for $\mtr {\hat{A}}\paren{M} $ defined in \eqref{EQN_modified_matrix} the following problem:
 \begin{equation}\label{EQN_problem_2_infinite}
\begin{aligned} 
&\min_{x\in\Rmax^{n}}&& \abs{\supp{x}}\\
&\textnormal{subject to}&&\norm{\vct b-\mtr {\hat{A}}\paren{M}\mxsmp \vct x}_{1}\le \epsilon\\
& &&\mtr {\hat{A}}\paren{M}\mxsmp \vct x\le \vct b
\end{aligned}
\end{equation}
is equivalent to problem~\eqref{problem_2}.
\end{lemma}

Now, we can  just apply Theorem \ref{Theorem_supermodular_bounds_finite} to the finite matrices $\mtr{\hat{A}}\paren{M}$ and $\vct b$. 
\begin{theorem} \label{Theorem_infinite_approximation_bounds}
Suppose $M>\epsilon\ge 0$ are constants such that $E\paren{\emptyset}>\epsilon$ and $E\paren{J}\le \epsilon$, where $E$ is defined in~\eqref{EQN_residual_error_norm}. Let $k$ be the time Algorithm~\ref{infinite_approximation_algorithm} terminates under input $\mtr{\hat{A}}\paren{M},\vct b$, where $\mtr{\hat{A}}\paren{M}$ is defined in~\eqref{EQN_modified_matrix}. Let $\vct {\hat{x}}$, $T_k$ be the respective outputs. Then,  $\vct {\hat{x}}$ is a suboptimal solution to problem~\eqref{problem_2} with $T_k=\supp{ {\hat{x}}}$. 
 Moreover, if $T^{*}=\supp{x^*}$, where $\vct x^*$ is an optimal solution of problem~\eqref{problem_2}, the following inequality holds:

\begin{equation}\label{infinite_bound_approximation_error}
\frac{\abs{T_{k}}}{\abs{T^{*}}}\le 1+\log\paren{\frac{m\Delta}{\min\set{E\paren{T_{k-1}},M}-\epsilon}}
\end{equation}
where $\Delta=\bigsup{i\in I,j\in J}{}\paren{b_{i}-\hat{A}_{ij}(M)-\bar{x}_{j}}$.
\hfill $\diamond$
\end{theorem}

\begin{proof}
Let $\hat{E}\paren{T}=\vct b-\bigsup{j\in T}{}\paren{\vct{\hat{A}}(M)_{j}+\bar{x}_{j}}$ be the $\ell_1$-error function for $\mtr{\hat{A}}(M),\,\vct b$. From Theorem~\ref{Theorem_supermodular_bounds_finite} and Lemma~\ref{LEM_New_Solution_Same}, we obtain:
\[
\frac{\abs{T_k}}{\abs{T^{*}}}\le 1+\log\paren{\frac{m\Delta}{\hat{E}\paren{T_{k-1}}-\epsilon}}
\]
But either $\hat{E}\paren{T_{k-1}}=E\paren{T_{k-1}}$ if there is no infinite component in $\vct{e}(T_{k-1})$, or  $\hat{E}\paren{T_{k-1}}\ge M$ if there is some infinite component in $\vct{e}(T_{k-1})$.
\end{proof}
\begin{remark}
Consider the notation of the previous theorem.	Notice that:
	\[
	\Delta=\max\{\bigsup{i\in I,j\in J,A_{ij}\neq-\infty}{}\paren{b_{i}-A_{ij}-\bar{x}_{j}},M\},
	\]
	where $M$ is used to replace the $-\infty$ elements in~\eqref{EQN_modified_matrix}.
By increasing $M$ we might make the nominator in~\eqref{infinite_bound_approximation_error} bigger. Thus, in the sufficient condition $M>\epsilon$ it might be a good choice to select $M$ close to $\epsilon$. On the other hand, we should not choose $M$ too close to $\epsilon$, since we might make the denominator small. In the case of integer valued elements, a reasonable selection could be $M=\epsilon+1$, since it guarantees $\frac{\abs{T_{N}}}{\abs{T^{*}}}\le 1+\log\paren{\frac{m\Delta}{\lfloor\epsilon+1\rfloor-\epsilon}}$ as in the finite element case.
\end{remark}

\section{Application to the Sparse Recovery Problem}\label{Section_Inversion}
In the recovery problem, the goal is to reconstruct an unknown vector $\vct z\in \Rmax^{n}$ from the measurements $\mtr A\mxsmp \vct z\in\REAL^{m}$, by solving the equation:
\begin{equation}\label{EQN_sparse_recovery_system}
\mtr A\mxsmp \vct x=\mtr A\mxsmp \vct z.
\end{equation}
If the equation $\mtr A\mxsmp \vct x=\mtr A\mxsmp \vct z$ has a unique solution then the principal solution can recover $\vct z$. Uniqueness holds only if the whole collection $\set{I_j:\:j\in J}$ is needed to cover $I$ (see Chapter~15 of \cite{Cuni79} or~\cite{Butk10} or Corollary~4.8 in~\cite{akian2005set}), where $I_j$ are defined in~\eqref{sets_which_cover}.  
In other words, the principal solution will recover $\vct z$ only if $\vct z$ is dense. 
If the original $\vct z$ is sparse then, in general, the equation $\mtr A\mxsmp \vct x=\mtr A\mxsmp \vct z$, will not have a unique solution and the principal solution will misidentify the $-\infty$ elements as finite.

Here, we explore conditions under which we could estimate a sparse $\vct z$ by computing $\vct{x}^{*}$, i.e. one of the sparsest solutions to problem~\eqref{problem_1}. We call this the sparse recovery problem. 
\begin{Problem}[Sparse Recovery]
	Consider an arbitrary vector $\vct z\in\Rmax^n$ such that the pair $\paren{\mtr A,\,\vct b}=\paren{\mtr A,\,\mtr A\mxsmp \vct z}$ satisfies Assumption~\ref{assuption_A}. Let $\vct x^*$ be the optimal solution  of problem~\eqref{problem_1} for the pair $\paren{\mtr A,\,\vct b}=\paren{\mtr A,\,\mtr A\mxsmp \vct z}$. We say that the Sparse Recovery Problem is solved if $\vct x^*$ recovers $\vct z$ or \[\vct z=\vct x^*.\]
\end{Problem}
This problem is also related to the system identification problem  \citep{2006SchullerusSchutter}, where, however, the sparsity patter is considered known.
Notice that in general there might be multiple sparsest solutions to the max-plus equation--see Example~\ref{EXAMPLE_sparse}. However, the sparse recovery problem above can only be solved exactly when the sparsest solution $\vct{x}^{*}=\vct z$ is unique\footnote{We note that uniqueness of the sparsest solution $x^{*}$ is different than the uniqueness of the equation $\mtr A\mxsmp \vct x=\vct b$. The former requires a unique minimum set-cover, while the later requires that the minimum set-cover is the whole collection $\set{I_j:\:j\in J}$.}. Even if $x^{*}$ is unique, it will have more $-\infty$ components than $\vct z$ in general. Nonetheless, under some sufficient conditions, the sparse recovery problem can be solved as the next theorem proves. 

\begin{theorem} \label{Theorem_inversion}
	Consider an element $\vct z\in \Rmax^{n}$ such that the pair $\paren{\mtr A,\,\mtr A\mxsmp \vct z}$ satisfies Assumption~\ref{assuption_A}. Let $\vct x^*$ be the optimal solution of problem~\eqref{problem_1} for $\paren{\mtr A,\,\vct b}=\paren{\mtr A,\,\mtr A\mxsmp \vct z}$. Then, $\vct x^*=\vct z$ if the following sufficient condition holds:
	For every finite component $j\in \supp{z}$, there exists a row index $i=i(j)\in I$ such that:
	\begin{enumerate}[a)]
		\item for all other indices in the support, $k\in \supp{z},\,k\neq j$, we have:
		\[A_{ij}>A_{ik}+z_{k}-z_{j}\]
		\item for all indices in the complement of the support, $l\in J\setminus\supp{z}$, there exists at least one row index $s=s(j,l)\in I$, such that:
		\[A_{sl}>A_{il}+[\mtr A\mxsmp \vct z]_{s}-[\mtr A\mxsmp \vct z]_{i}.\]
	\end{enumerate} 
\end{theorem}
Intuitively, the first part of the condition of the preceding theorem states that for any component $j\in J$ with $z_{j}\neq -\infty$, there must be at least one row index $i$ for which $A_{ij}$ is large enough, in order to observe the influence of $z_{j}$ in $\mtr A \mxsmp \vct z$. Given the previous pair $(i,j)$, the second part of the condition requires that for every $l\in J\setminus\supp{z}$, there exists some row $s\in I$ such that the component $A_{sl}$ is large enough to reveal that $z_{l}$  is smaller than $z_{j}$; small enough to be  $-\infty$. 

Both conditions can be guaranteed if, for example, $m\ge n$ and $\mtr A$ has large enough leading diagonal elements (or large diagonal elements up to permutations--see Section~\ref{Section_Example}). In this case, if $A_{jj}$, $j\in\supp{z}$, is large enough across the $j$-th row  then part a) is satisfied with $i=j$. Similarly if $A_{ll}$, $l\in J\setminus\supp{z}$, is large enough across the $l$-th column, then part b) is satisfied by choosing $s(j,l)=l$ for all $j\in \supp{z}$.
\section{Applications}\label{Section_Applications}
In this section, we give several applications of the present framework. First, we provide two possible applications in discrete-event systems: i) resource optimization; and  ii) system identification with unknown sparsity pattern. Then, we show how the pruning problem can be formulated as a sparsity problem.
\subsection{Discrete Event Systems}
 We motivate the application to discrete-event systems through multi-machine interactive production processes \citep{Butk10}. Consider $m$ different products, which are made using $n$ machines. 
A machine $j\in J$ contributes to the completion of a product $i\in I$ by making a partial product. It processes all partial products in parallel as soon as it starts working. A system matrix $\mtr G\in \Rmax^{m\times n}$ determines how much time it takes for the partial products to be made. Each element $G_{ij}$ represents the time needed for machine $j$ to make the partial product for product $i$. Thus, either $G_{ij}\ge 0$ or $G_{ij}=-\infty$ if product $i$ does not depend on machine $j$. An input $\vct u\in R^{n}$ describes the times the machines start working; $u_{j}$ is the time, at which machine $j$ starts working. If $u_j=-\infty$, then the machine $j$ is not used at all. The output 
\begin{equation}\label{EQN_mmip_output}
\vct y=\mtr G\mxsmp \vct u
\end{equation} describes the times the products are made; product $i$ is completed at time $y_{i}$. We will use the above model to explore the following problems.

\subsubsection{Resource Optimization}\label{Section_Applications_Resource}
Suppose that the products have  delivery deadlines $\vct d\in \REAL^m$, which should not be exceeded. This implies that the outputs $\vct y$ should satisfy the lateness constraint $\vct y\le \vct d$. Meanwhile, it costs storage resources to make the products before the delivery time. Thus, we wish to restrict the earliness $\norm{\vct d-\vct y}_1$. 
 Suppose now that we have an extra constraint; we also want to minimize the number of machines used, which consume energy resources.
 Recall that when $u_j=-\infty$, then machine $j$ is not used. Thus, the number of used machines is equal to the cardinality of the support of vector $\vct u$.
This problem could be formulated as an instance of problem~\eqref{problem_2} with $\mtr A=\mtr G$, $\vct x=\vct u$, $\vct b=\vct d$. Sparsity here implies resource efficiency, since we use fewer machines.
Notice that in this case the lateness constraint is not a relaxation but a desired property.
\subsubsection{Structure-seeking System Identification}
Assume we have an unknown system matrix $\mtr G\in \Rmax^{m\times n}$. Our goal is to recover $\mtr G$ from a sequence of $K$ input-output pairs $\paren{\vct u_{l}, \vct y_{l}}\in \Rmax^{n}\times \REAL^{m},\: l\in L=\set{ 1,\dots,K}$. Those pairs are related via the max-plus model~\eqref{EQN_mmip_output}: $\vct y_{l}=\mtr G \mxsmp \vct u_{l},\: l\in L$ (we assume the output is finite). If we stack the inputs and outputs together, we obtain a set of max-plus equations:
\[
\underbrace{\matr{{c}\vct y_{1}^{\intercal}\\ \vdots\\ \vct y_{K}^{\intercal}}}_{\mtr Y}=\underbrace{\matr{{c}\vct u_{1}^{\intercal}\\ \vdots\\ \vct u_{K}^{\intercal}}}_{\mtr U} \mxsmp\mtr G^{\intercal}
\]
or 
\begin{equation}\label{EQN_estimation_equation}
\mtr Y=\mtr U \mxsmp \mtr G^{\intercal}.
\end{equation}
Notice that $\mtr Y \in \REAL^{K\times m}$, $\mtr U \in \Rmax^{K\times n}$.
System~\eqref{EQN_estimation_equation} consists of $m$ separate max-plus equations written together in matrix form.

In this scenario, the infinite elements of $\mtr G$ reflect the structure of the system. As mentioned before, $G_{ij}=-\infty$ means that the product $i$ does not depend on the machine $j$. 
Here, we are interested in obtaining a solution that not only solves the above equation but also reveals the system structure. (We assume that we do not have any a priori knowledge about the structure of system $\mtr G$; the only information is input-output pairs.)

Without any sparsity constraints, the principal solution $\mtr{\bar{G}}$ will have only finite elements, hiding the actual sparsity pattern of the original system matrix $\mtr G$. Thus, we have to find another way to identify the $-\infty$ elements. 
One way to approach this problem would be to solve the sparse recovery problem instead. 
If the sufficient conditions of Theorem~\ref{Theorem_inversion} are satisfied, then exact reconstruction is possible. In fact, those conditions also suggest a way to do experiment design, i.e. to design the inputs $\mtr U$.  It is sufficient to select $\mtr U$ with large enough leading diagonal elements such that the sparsest solution recovers $\mtr G$. Without knowing $\mtr G$, we may not be able to compute how large the leading diagonal elements should be. Nonetheless, we could overcome this problem by exploiting bounds on the finite elements of $\mtr G$.

\subsection{Pruning}
The pruning problem emerged as a curse-of-dimensionality-free method for approximating optimal control value functions--see~\cite{mceneaney2009pruning},~\cite{gaubert2011curse} for more details and motivation behind the method. Next, we show that a ``dual" version of the pruning problem can be formulated in terms of a sparsity problem as in~\eqref{problem_2}.

 Suppose that $\vct \phi=\bigsup{j=1}{n}{ \vct{\phi_j}}$, where $\vct{\phi_j}\in \Rmax^{m}$ are basis vectors. The goal is to find a reduced representation $\vct{\tilde{\phi}}=\bigsup{j\in S\subseteq J}{}{ \vct{\phi_j}}$ such that the approximation error $\norm{\vct \phi-\vct{\tilde{\phi}}}_{\ell_1}$ is small. 
Let $\vct x\in\Rmax^n$ indicate which basis columns should be selected. If $x_{j}=-\infty$, then we ignore column $\vct{\phi_j}$, otherwise we select it. To solve the pruning problem, we could formulate it as an $\epsilon-$approximate sparsity problem:	
\begin{align*}
&\min_{x\in\Rmax^{n}}&& \abs{\supp{x}}\\
&\text{subject to}&&\norm{\vct \phi-\matr{{ccc}\vct{\phi_1}&\dots&\vct{\phi_n}}\mxsmp \vct x}_{1}\le \epsilon\\
& &&\matr{{ccc}\vct{\phi_1}&\dots&\vct{\phi_n}}\mxsmp \vct x\le \vct \phi.
\end{align*}
The formulation here is a ``dual" version of the one that appears in~\cite{gaubert2011curse}. There, the minimization is with respect to $\norm{\vct \phi-\matr{{ccc}\vct{\phi_1}&\dots&\vct{\phi_n}}\mxsmp \vct x}_{1}$, while the cardinality of the support $\abs{\supp{x}}=k$ is kept fixed to a value $k$.

\section{Numerical Examples and Simulations}\label{Section_Example}
\subsection{System Identification}
In this subsection, we present a numerical example, where we apply our results to the system identification problem. 
We implemented the greedy Algorithm~\ref{infinite_approximation_algorithm} in Matlab to obtain solutions to problems~\eqref{problem_1},~\eqref{problem_2}.
 For the small examples below, we can verify by hand that the greedy solution will also be optimal.

 Consider a multi-machine interactive production process, as defined in Section~\ref{Section_Applications}, with system matrix $\mtr G\in\Rmax^{m\times n}$
\[ 
\mtr G=\matr{{ccc} 2& 3& -\infty\\1&1&-\infty\\-\infty&2&6}
\] 
Now consider $K=4$ input instances $\vct u_{l}$, $l=1,\dots,4$, which are imposed to the system:
\[
\mtr U^{\intercal}=\matr{{ccc}\vct u_{1}&\dots&\vct u_{4}}=\matr{{ccccc} 0   &  10  &    0   &   2  \\
    10   &   0   &   0   &   0   \\
     5 &     5  &   10   &   2  \\}
\].
 For each input instance $\vct u_{l}$, the respective outputs $\vct y_{l}$, $l=1,\dots,4$ are:
 
 \[
\mtr Y^{\intercal}=\matr{{ccc}\vct y_{1}&\dots&\vct y_{4}}=\mtr G\mxsmp\mtr U^{\intercal}=\matr{{cccc} 13&12&3&4\\
    11&11&1&3\\
     12&11&16&8}.
\]
Our goal is to determine $\mtr G$ from the given inputs and the corresponding outputs. 
The principal solution gives:
\[
\mtr{\bar G}=\matr{{ccr}2&3&	-7\\
1	&1	&-9\\
1&	2	&6},
\]
which hides the sparsity pattern of the original matrix $\mtr G$. Notice that some $-\infty$ elements in $\mtr G$, i.e. $G_{13}$, correspond to negative elements in $\mtr{\bar G}$, i.e. $\mtr{\bar G}_{13}=-7$. Those can be identified as $-\infty$, since $\mtr G$, can only have positive or $-\infty$ elements. However, not all of them are negative. For instance, $\mtr{\bar G}_{31}=1$. Thus, this method does not guarantee that all $-\infty$ elements are revealed.

Suppose now that we compute the sparsest solution $\mtr{G^{*}}$, by solving problem~\eqref{problem_1}. In this case, we obtain:
\[\mtr{G^{*}}=\mtr G.\]
This result is expected, since $\mtr U$ is designed to have large diagonal values under the column permutation $\set{2,1,3}$ and satisfies the assumptions of Theorem~\ref{Theorem_inversion}. Thus, without any prior knowledge, we managed to identify for all products, which machines they depend on. 
If this condition is not satisfied, i.e. if we change $U_{12}$ from $10$ to $U_{12}=1$, then the sparsest solution falsely yields
$G^*_{32}=-\infty\neq G_{32}$, but it correctly identifies the remaining elements.

For the next example, suppose that due to some unexpected delay the last output is $\vct y_4=\matr{{ccc}4.2&3&8}^{\intercal}$. The equation $\mtr Y=\mtr U \mxsmp \mtr G^{\intercal} $ is no longer satisfied. In this case, we solve problem~\eqref{problem_2} and find the sparsest approximate solution $\mtr{\hat{G}}$.  For $\epsilon=0.3$, we have $\mtr{\hat G}=\mtr G$ and we recover $\mtr G$. However, if the error gets bigger, for example $Y_{41}=5$,  the sparsest approximate solution falsely returns $G^{*}_{32}=-\infty$ for $\epsilon=1$.  
The results for the sparse recovery problem, presented in Section~\ref{Section_Inversion}, are only applicable to the exact solution case. Nonetheless, from the last numerical example, it seems that if the delay is small, they might still be valid for the approximate solution case.  
It is subject of future work to provide a formal analysis.

\subsection{Greedy Algorithm Performance}\label{Subsection_simulations}
In this subsection, we explore the performance of the greedy Algorithm~\ref{infinite_approximation_algorithm} with respect to problem~\eqref{problem_2}. First, we construct an example where the greedy Algorithm~\ref{infinite_approximation_algorithm} is suboptimal. Then, we compare Algorithm~\ref{infinite_approximation_algorithm} with the brute force one, using random matrices $\mtr A$, $\vct b$. For the brute force algorithm, we solve a combinatorial problem; we search over all possible supports $\supp{x}$.
\begin{exampl}[Suboptimality of greedy algorithm]
	\normalfont 
Consider the matrices:
\[
\mtr A=\matr{{rrr}0&0&-10\\-2&0&-10\\-2&-10&0},\,\vct b=\matr{{c}0\\0\\0}
\]
and let $\epsilon=1$. The optimal solution to problem~\eqref{problem_2} is $\vct x^*=\matr{{ccc}-\infty&0&0}^{\intercal}$.
The greedy algorithm will initially select $T_1=\set{1}$, since the first column of $A$ leads to the smallest error. However, in this example, it is sufficient and necessary for both components $2,3$ to be included in the support in order to have error less than $\epsilon$. Hence, the greedy algorithm will return the set $T_3=\set{1,2,3}$ and the suboptimal solution $\vct{\hat x}=\matr{{ccc}0&0&0}^{\intercal}$.	\hfill $\diamond$
\end{exampl}

 Next, we compare the greedy algorithm with the brute force one. Both were implemented in Matlab. For the comparison we compute the suboptimality ratio of the greedy algorithm as well as the execution times. Due to the exponential complexity of the brute force algorithm, this comparison can only be made for small values of $n$, where $n$ is the number of columns of matrix $\mtr A$. 
 
We generated random $m\times n$ matrices $\mtr A$ with elements taking values in the set $\set{0,\dots,n-2}$ and $m\times 1$ vectors $\vct b$ with elements taking values in $\set{0,\dots,n+5}$, for several $(m,n)$ pairs--see Table~\ref{Table_greedy_comparison}. Because the times and the suboptimality ratios depend on the randomly sampled matrices, we averaged them over $40$ independent iterations for each $(m,n)$ pair. To guarantee feasibility, in all of the cases we selected $\epsilon=\norm{\vct b-\mtr A\mxsmp \vct{\bar{x}}}+1$, where $\bar{x}$ is the principal solution. 
We observe that the average suboptimality ratio of the greedy algorithm is very close to one and does not decrease noticeably. Meanwhile, as we expected, the execution time of the greedy algorithm scales much better than the brute force one. Thus, empirically the greedy algorithm performs very well on average for small values of $m,n$. As we stated above, it is not easy to empirically evaluate the performance for larger values of $m,n$ since the brute force algorithm requires a lot of time to terminate.

\begin{table}[t]
	\begin{center}
\begin{tabular}{|c|c|c|c|c|c|c|c|}
	\hline
	$(m,n)$ &(8,16)& (8,17)&(9,18)&(9,19)&(10,20)&(10,21)&(11,22) \\
	\hline
suboptimality ratio  & 0.970& 0.948 & 0.952& 0.968 & 0.967& 0.955 & 0.979\\
	\hline
time greedy (sec) & 0.0012& 0.0013 & 0.0015& 0.0017 & 0.0019& 0.0020 &0.0022 \\
time brute force (sec)  & 0.09& 1.33 & 2.72& 5.56 & 11.30& 22.37 & 46.73\\
	\hline
\end{tabular}
\caption{Comparison between the greedy and the brute force algorithm for random matrices $\mtr A$, $\vct b$. For every pair of $(m,n)$, the average is over $40$ independent samples. The greedy algorithm performs very well on average for small $(m,n)$. It has suboptimality ratio close to one and is much faster than the brute force algorithm.}\label{Table_greedy_comparison}
\end{center}
\end{table}

\section{Conclusion} \label{Section_Conclusion}
We studied the problem of finding the sparsest solution of the max-plus equation and proved that it is NP-complete. Then we studied the problem of finding the sparsest approximate solution subject to a lateness constraint. The degree of approximation was measured via a $\ell_{1}$ norm function, which was proved to have supermodular properties. Thus, we developed a greedy algorithm of polynomial complexity,  which approximates the optimal solution with guaranteed ratio of approximation. We also derived sufficient conditions such that the sparse recovery problem can be solved. The present framework can be applied to discrete event systems applications such as resource optimization or system identification. 
In future work we will explore whether we can drop the lateness constraint when searching for the sparsest approximate solution. We will also study whether the sufficient conditions of the sparse recovery problem can be relaxed.
Another direction is extending the concepts of sparsity to max-plus dynamical systems.
Finally, we would like to extend the results to more general idempotent semi-rings by using residuation theory.
\section*{Appendix A: Previous Results}\label{Appendix_A}

The result below was originally proved in~\cite{vorobyev1967extremal} and~\cite{zimmermann1976extremal}. A reference in English can be found in~\cite{Butk03}.
\begin{theorem} [Covering theorem]\label{theorem_multiple_solutions}
	An element $\vct x\in \Rmax^{n}$ is a solution to \eqref{definition_1} or $\vct x\in S\paren{\mtr A,\vct b}$ if and only if:
	\begin{align*}
	\text{a) }& \vct x\le \bar{\vct x}\\
	\text{b) }&\bigcup_{j\in J_{x}} I_{j}=I,
	\end{align*} 
	where $\vct{\bar{x}}$ is the principal solution defined in \eqref{principal_solution}, set $J_{x}$ is defined in~\eqref{EQN_agreement_set}, and sets $I_j$ are defined in \eqref{sets_which_cover}.
	\hfill $\diamond$ 
\end{theorem}

\begin{theorem}[\cite{Cuni79}]\label{Theorem_approximate_max_plus_solution}
	Let $\bar{\vct x}$ be the principal solution defined in \eqref{principal_solution}. The following equivalence holds:
	\begin{equation}\label{EQN_equivalence_of_inequalities}
	\mtr A \mxsmp \vct x\le \vct b \Leftrightarrow \vct x\le \vct{\bar{ x}}.
	\end{equation}
	Moreover, $\bar{\vct x}$ is an optimal solution to problem~\eqref{EQN_approximate}. \hfill $\diamond$ 
\end{theorem}

\section*{Appendix B: Proofs}
\subsection*{Proof of Theorem \ref{basic_theorem_1}}
First, we prove i).	Suppose that $\vct x^{*}$ is an optimal solution to \eqref{problem_1}. Since it is a solution of the equation $\mtr A\mxsmp \vct x=\vct b$, by Theorem~\ref{theorem_multiple_solutions}, the subcollection $\set{I_{j}:\:j\in J_{x^{*}}}$, determined by the agreement set $J_{x^{*}}=\set{j\in J: x^{*}_{j}=\bar{x}_{j}}$, is a set cover of $I$. We will show that the size $\abs{J_{x^{*}}}$ of the set cover is minimum. By optimality of $\vct x^*$, we necessarily have $x^{*}_{j}=-\infty$, for $j\in J\setminus J_{x^{*}}$ and the support of $\vct x^{*}$ is the same as the agreement set; otherwise, we could create a sparser solution by forcing the elements outside of the agreement set to be $-\infty$. So, $\abs{\supp{x^{*}}}=\abs{J_{x^{*}}}$. Now, take any set cover $\set{I_{j}:\:j\in K\subseteq J}$ of $I$  and define element $\vct x\paren{K}$ as:
\begin{equation}\label{set_cover_sparse}
\begin{aligned}
\vct x\paren{K}_{j}&=\bar{x}_{j},\: j\in K\\
\vct x\paren{K}_{j}&=-\infty,\: j\in J\setminus K
\end{aligned}
\end{equation}
Notice that  $\abs{\supp{x\paren{K}}}=\abs{K}$ and by Theorem~\ref{theorem_multiple_solutions}, $\vct{x}\paren{K}$ is also a solution to the max-plus equation $\mtr A\mxsmp \vct x=\vct b$. By optimality, $x^{*}$ has the smallest support, or $\abs{\supp{x^{*}}}\le \abs{\supp{x\paren{K}}}$. But this implies that $\abs{J_{x^{*}}}\le\abs{K}$, which shows that $\set{I_{j}:\:j\in J_{x^{*}}}$ is a minimum set cover of $I$.

Conversely, suppose the collection $\set{I_{j}:\:j\in K^*\subseteq J}$ is a minimum set cover of $I$. Then, we can define the solution $\vct {\hat{x}}$ as in~\eqref{EQN_sparse_solution}. We will show that $\vct {\hat{x}}$ is an optimal solution to~\eqref{problem_1}. Suppose $\vct x^{*}$ is one optimal solution to~\eqref{problem_1}. Then, the collection $\set{I_{j}:\:j\in J_{x^{*}}}$ is a set cover with $J_{x^{*}}=\set{j\in J: x^{*}_{j}=\bar{x}_{j}}$. Since $\vct x^*$ is the sparsest solution, we can only have $\abs{\supp{x^{*}}}=\abs{J_{x^{*}}}$. Meanwhile, by optimality of the set cover we have \[\abs{\supp{\hat{x}}}=\abs{K^*}\le \abs{J_{x^{*}}}=\abs{\supp{x^{*}}}.\] Hence, $\vct {\hat{x}}$ is also an optimal solution to~\eqref{problem_1}.

Second, we prove ii). This part is adapted from \cite{Butk03}. Suppose we are given an arbitrary collection of nonempty subsets \[S_j\subseteq\set{1,\dots,m}=I,\,j\in\set{1,\dots,n}=J,\] 
for some $m,\,n\in\mathbb{N}$, such that $\bigcup_{j\in J}S_j=I$. Define $A_{ij}=\1\paren{i\in S_j}$ for all $i\in I,\,j\in J$, where $\1$ is the indicator function, and $b_i=1$, for all $i\in I$. 
By equations \eqref{principal_solution}, \eqref{sets_which_cover}, it follows that the principal solution is $\vct{\bar{x}}=\matr{{ccc}1&\dots&1}^{\intercal}$, while the sets $S_j$ are equal to the sets $I_j$. But following the analysis of i), finding the minimum set cover of $I$ using $S_j=I_j$ is equivalent to finding the solution to problem~\eqref{problem_1} with the above selection of $\mtr A,\vct b$. This completes the proof of part~ii).
\hfill$\qed$
\subsection*{Proof of Lemma~\ref{Proposition_principal_fixed_support}}
By construction, the agreement set and the support are equal to $T$ or
\begin{align*}
z_{j}&=\bar{x}_j,\text{ for }j\in T\\
z_{j}&=-\infty,\text{ for }j\in J\setminus T
\end{align*}
Thus, $\vct z\le \vct{\bar{x}}$ and by Theorem~\ref{Theorem_approximate_max_plus_solution}, also $\mtr A\mxsmp \vct z\le \vct b$, which proves that $\vct z\in X_{T}$. 

To prove the second part, again from Theorem~\ref{Theorem_approximate_max_plus_solution}, if $\vct x\in X_{T}$ then \begin{align*}
x_{j}&\le \bar{x}_{j}=z_{j}, \text{ for }j\in T\\
x_{j}&=z_{j}=-\infty, \text{ for }j\in J\setminus T
\end{align*}
 As a result, $\vct x\le \vct{z}$ for any $\vct x\in X_{T}$. Now, since $\mtr A\mxsmp \cdot$ is increasing  \citep{Cuni79} we obtain the inequality:
\[\vct b -\mtr A\mxsmp \vct z\le \vct b -\mtr A\mxsmp \vct x,\] for any $\vct x\in X_{T}$.  Since both $\vct x, \vct z$ satisfy the lateness constraint~\eqref{lateness_constraint}, we finally have \[\norm{\vct b-\mtr A\mxsmp \vct x}_{1}=\vct 1^{\intercal}\paren{\vct b -\mtr A\mxsmp \vct x}\ge \vct 1^{\intercal}\paren{\vct b -\mtr A\mxsmp \vct z}=\norm{\vct b-\mtr A\mxsmp \vct z}_{1}\] for any $\vct x\in X_{T}$, where $\vct 1=\matr{{ccc}1& \cdots&1}^{\intercal}$.
\hfill$\qed$
\subsection*{Proof of Corollary~\ref{COR_AlgorithmsEquivalent}}
Let $x^{*}$, $\hat{T}$ be the optimal solutions of problems \eqref{problem_2}, \eqref{EQN_reformulated_problem} respectively. Denote by $T^{\star}=\supp{x^{*}}$ the support of $\vct x^*$. Then construct a new vector $\vct z^{*}$ such that $z^{*}_{j}=\bar{x}_j,\,j\in T^{*}$ and $z^{*}_{j}=-\infty,\,i\in J\setminus T^{*}$. By Lemma~\ref{Proposition_principal_fixed_support}, \[E\paren{T^{*}}=\norm{\vct b-\mtr A\mxsmp \vct z^{*}}_{1}\le \norm{\vct b-\mtr A\mxsmp \vct x^{*}}_{1}\le \epsilon.\] 
Thus, $T^{*}=\supp{x^{*}}$ is a feasible point of problem \eqref{EQN_reformulated_problem}, implying $|\hat{T}|\le\abs{T^{*}}$. 

Conversely, define $\vct {\hat{x}}$ as in \eqref{EQN_reconstructed_solution_from_reformulation}. By construction and the feasibility of $\hat{T}$ and Lemma~\ref{Proposition_principal_fixed_support}, we have: 
\begin{align*}
\norm{\vct b-\mtr A\mxsmp \vct{\hat{x}}}_{1}&=E(\hat{T})\le \epsilon.\\
\mtr A\mxsmp \vct{\hat{x}}&\le \vct b
\end{align*}
Thus, $\vct{\hat{x}}$ is a feasible point of problem~\eqref{problem_2}, which implies $\abs{T^*}\le |\hat{T}|$. From the above inequalities we obtain $\abs{T^*}= |\hat{T}|$, which also proves that $\vct{\hat{x}}$ is an optimal solution to problem~\eqref{problem_2}.
\hfill$\qed$
\subsection*{Proof of Theorem~\ref{Supermodular_error_function}}
	Notice that we have:
	\[
\bigsup{j\in T}{}\paren{\vct A_{j}+\bar{x}_{j}}\le \bigsup{j\in J}{}\paren{\vct A_{j}+\bar{x}_{j}}=\mtr A \mxsmp \vct{\bar{x}}\le \vct b	
	\]
	Thus, we get by construction that the error vector $\vct e(T)$ has only positive components, for every $T\subseteq J$, which implies: \begin{equation}\label{EQN_positivity_of_error}
	E\paren{T}=\norm{\vct e(T)}=\vct 1^{\intercal}\vct e(T),
	\end{equation}
	where $1^{\intercal}=\matr{{ccc}1&\dots&1}^{\intercal}$.
	For convenience, define matrix $ \mtr{\hat{ A}}\in \Rmax^{m\times n}$ as $\hat{A}_{ij}=A_{ij}+\bar{x}_{j}$. 
 Then, by the definition~\eqref{EQN_residual_error_vector} of error vector: 
 \[ \bigsup{j\in T}{}\vct{\hat{A}_{j}}=\vct b-\vct e\paren{T}.\] 

First, we show that $E\paren{T}$ is decreasing. Let $B$, $C$ be two nonempty subsets of $J$ with $C\subseteq B\subset J$. Then, $\bigsup{j\in C}{}\vct{\hat{A}_{j}}\le \bigsup{j\in B}{}\vct{\hat{A}_{j}}$. Consequently, $e\paren{B}\le e\paren{C}$.
Now if $C$ is empty and $B$ is non-empty, then by construction $e\paren{\emptyset}\ge\bigsup{k\in J}{}e\paren{\set{k}}\ge e\paren{B}$  (if $C$, $B$ are both empty we trivially have $e\paren{C}=e\paren{B}$).
In any case, by \eqref{EQN_positivity_of_error}, we obtain $E\paren{C}\ge E\paren{B}$.

Second, we show that $E\paren{T}$ is supermodular. Let $C\subseteq B\subseteq J$ and $k\in J\setminus B$. It is sufficient to prove that:
\begin{equation}\label{EQN_Thm_supermodular_goal}
\vct e(C\cup\set{k})-\vct e(C)\le \vct e(B\cup\set{k})-\vct e(B).
\end{equation}
For $C\neq \emptyset$ define: 
\begin{align*}
\vct u&=\bigsup{j\in C}{}\vct{\hat{A}_{j}},&&\,\vct v=\bigsup{j\in C\cup \set{k}}{}\vct{\hat{A}_{j}}\\
\vct z&=\bigsup{j\in B}{}\vct{\hat{A}_{j}},&&\,\vct w=\bigsup{j\in B\cup \set{k}}{}\vct{\hat{A}_{j}}.
\end{align*}
By this definition, $v_{i}=u_{i}\ssup \hat{A}_{ik}$, $w_{i}=z_{i}\ssup \hat{A}_{ik}$ for every $i\in I$. Also, by monotonicity $\vct u\le \vct z$, $\vct v\le \vct w$. There are three possibilities:
\begin{enumerate}[i)]
\item If $u_{i}>\hat{A}_{ik}$ then $v_{i}=u_{i}$. But also $w_{i}=z_{i}$, since by monotonicity $\vct z\ge \vct u$ and $z_{i}\ge u_{i}> \hat{A}_{ik}$. In this case, $v_{i}-u_{i}=w_{i}-z_{i}=0$.
\item If $u_{i}\le \hat{A}_{ik}$ and $z_{i}>\hat{A}_{ik}$ then $v_{i}-u_{i}=\hat{A}_{ik}-u_{i}\ge 0$ and $w_{i}-z_{i}=0\le v_{i}-u_{i}$. 
\item If both $u_{i}\le \hat{A}_{ik}$ and $z_{i}\le \hat{A}_{ik}$ then $v_{i}-u_{i}=\hat{A}_{ik}-u_{i}\ge \hat{A}_{ik}-z_{i}=w_{i}-z_{i}$, since by monotonicity $z_{i}\ge u_{i}$. 
\end{enumerate}
If $C$ is the empty set, we define $\vct u=\vct b-\vct e(\emptyset)$ and $\vct u,\vct z,\vct w$ are defined as before. Since by construction $\vct e(\emptyset)\le \vct e({k})$ for all $k\in J$, we also have $\vct u\le \vct v$ and $\vct u\le \vct z$. Thus, either case ii) or case iii) apply.

In any case, $\vct v-\vct u\ge \vct w-\vct z$ which is equivalent to \eqref{EQN_Thm_supermodular_goal}. 
Finally, multiplying both sides of \eqref{EQN_Thm_supermodular_goal} from the left by $\vct 1^{\intercal}$ gives the desired result: $
E\paren{C\cup\set{k}}-E\paren{C}\le E\paren{B\cup\set{k}}-E\paren{B}
$.\hfill $\qed$
\subsection*{Proof of Theorem~\ref{Theorem_supermodular_bounds_finite}}
Define the truncated error set function \[\bar{E}\paren{T}=\max\paren{E\paren{T},\epsilon}.\] By Theorem \ref{Supermodular_error_function}, the error set function $E(T)$ is supermodular and decreasing. Thus, so is the truncated error function  \citep{2012Krause}. This enables as to express the constraint $E\paren{T}\le \epsilon$ as $\bar{E}\paren{T}=\bar{E}\paren{J}$. Then, the lines $6-11$ of Algorithm \ref{infinite_approximation_algorithm} are a version of Algorithm~\ref{Algorithm_preliminaries}. Hence, Theorem \ref{Theorem_supermodular_bounds} readily applies giving the bounds
\[
\frac{\abs{T_{k}}}{|\hat{T}|}\le 1+\log\paren{\frac{\bar{E}\paren{\emptyset}-\bar{E}\paren{J}}{\bar{E}\paren{T_{k-1}}-\bar{E}\paren{J}}},
\] 
where $\hat{T}$ is the optimal solution of problem~\eqref{EQN_reformulated_problem}. From Corollary~\ref{COR_AlgorithmsEquivalent}, we can replace $\hat{T}$ with $T^*$.
By the assumption $E(\emptyset)>\epsilon$ and the  definition of the $\ell_1$-error set function at $\emptyset$: \[\bar{E}\paren{\emptyset}=E\paren{\emptyset}=\sum_{i\in I}\bigsup{j\in J}{}\paren{b_{i}-A_{ij}-\bar{x}_{j} }\le m\Delta.\] Meanwhile, we have $\bar{E}(J)\ge 0$ and the result for the nominator in the logarithm follows.
For the denominator, notice that $k$ is such that $E\paren{T_{k-1}}> \epsilon$ and $E\paren{T_k}\le \epsilon$. Such $k$ exists since  $E\paren{J}\le \epsilon$ and in the worst case, Algorithm \ref{infinite_approximation_algorithm} halts at $k=\abs{J}$ with $T_{k}=J$. Thus, we have $\bar{E}\paren{T_{k-1}}=E\paren{T_{k-1}}$ and $\bar{E}\paren{J}=\epsilon$.
\hfill $\qed$

\subsection*{Proof of Lemma~\ref{LEM_New_Solution_Same}}
It is sufficient to prove that the feasible regions of both problems are identical.
First, we prove that:
\begin{equation}\label{EQN_New_Matrix_Equiv_One}
\mtr A\mxsmp \vct x\le \vct b \Leftrightarrow \mtr{\hat{A}}(M)\mxsmp \vct x\le \vct b
\end{equation}
But from Theorem~\ref{Theorem_approximate_max_plus_solution}, it is equivalent to show that $\vct{\bar{x}}=\vct{\hat{\bar{x}}}$, where $\vct{\bar{x}}$ is the original principal solution defined in \eqref{principal_solution} and $\vct{\hat{\bar{x}}}$ is the new principal solution with $\mtr{\hat{A}}(M)$ instead of $\mtr A$:
\begin{equation}\label{EQN_mod_principal_solution}
\hat{\bar{x}}_{j}=\biginf{i=1}{m}\paren{b_{i}-\hat{A}_{ij}(M)},\,\forall j\in J.
\end{equation}
By construction, $A_{ij}\le \hat{A}_{ij}(M)$, which by \eqref{principal_solution}, \eqref{EQN_mod_principal_solution}, implies $\vct{\hat{\bar{x}}}\le \vct{\bar x}$. To show the other direction, we have
\[
\hat{\bar{x}}_{j}=b_{k}-\hat{A}_{kj}(M), \text{ for some }k\in I.
\]
There are two cases:
\begin{enumerate}[i)]
	\item $\hat{A}_{kj}(M)=A_{kj}$. Then, $\hat{\bar{x}}_{j}=b_{k}-A_{kj}\ge \biginf{i\in I}{}b_{i}-A_{ij}= \bar{x}_j$.
	\item $\hat{A}_{kj}(M)=b_k-M-\bar{x}_j$. Then, $\hat{\bar{x}}_{j}=M+\bar{x}_j> \bar{x}_j$, since $M>0$.
\end{enumerate}
Thus, we also have $\vct{\hat{\bar{x}}}\ge \vct{\bar x}$. This proves $\vct{\bar{x}}=\vct{\hat{\bar{x}}}$.

Second, we prove that under the constraint $\mtr A\mxsmp \vct x\le \vct b$ (which we showed is equivalent to $\mtr{\hat{A}}(M)\mxsmp \vct x\le \vct b$) we have:
\begin{equation}\label{EQN_New_Matrix_Equiv_Two}
\norm{\vct b-\mtr A\mxsmp \vct x}_1\le \epsilon \Leftrightarrow \|\vct b-\mtr{\hat{A}}(M)\mxsmp \vct x\|_1\le \epsilon
\end{equation}

\noindent``$\Rightarrow$" direction. Since $A_{ij}\le \hat{A}_{ij}(M)$, we obtain 
\[\mtr A\mxsmp \vct x\le \mtr{\hat{A}}(M)\mxsmp \vct x .\] But we have $\mtr A\mxsmp \vct x\le \vct b,\:\mtr{\hat{A}}(M)\mxsmp \vct x\le \vct b.$
Thus,
\[\|\vct b-\mtr{\hat{A}}(M)\mxsmp \vct x\|_1 \le \norm{\vct b-\mtr A\mxsmp \vct x}_1\le \epsilon.\]
\noindent``$\Leftarrow$" direction. For every $i\in I$, there exists an index $j_i\in J$ such that:
\[
\|\vct b-\mtr{\hat{A}}(M)\mxsmp \vct x\|_1\\
=\sum_{i=1}^{m}(b_{i}-\hat{A}_{ij_{i}}(M)-x_{j_i})
\]
Now assume that some element $\hat{A}_{kj_{k}}(M)$ is equal to $-M+b_i-\bar{x}_{j_i}$, for some $k\in I$. Then, this implies
\begin{align*}
\epsilon&\ge \|\vct b-\mtr{\hat{A}}(M)\mxsmp \vct x\|_1=\sum_{i=1}^{m}(b_{i}-\hat{A}_{ij_{i}}(M)-x_{j_i})\\
&\ge b_{k}-\hat{A}_{kj_{k}}(M)-\bar{x}_{j_k}=M,
\end{align*}
where the second inequality follows from $\mtr{\hat{A}(M)}\mxsmp \vct x\le\vct b$ and the equivalent fact $\vct x\le \vct{\bar{x}}$ (see Theorem~\ref{Theorem_approximate_max_plus_solution}). But since $M>\epsilon$, this is a contradiction and the only possible case is $\hat{A}_{ij_{k}}(M)=A_{ij_i}$, for all $i\in I$. Finally, 
\begin{align*}
\epsilon&\ge\|\vct b-\mtr{\hat{A}}(M)\mxsmp \vct x\|_1=\sum_{i=1}^{m}(b_{i}-A_{ij_{i}}-x_{j_i})\\
&\ge  \sum_{i=1}^{m}(b_{i}-\bigsup{j\in J}{}(A_{ij}+x_j))=\norm{\vct b-\mtr A\mxsmp \vct x}_1.
\end{align*}
This completes the proof.
\hfill \qed

\subsection*{Proof of Theorem~\ref{Theorem_inversion}}
Define $\vct b=\mtr A\mxsmp \vct z$.
It is sufficient to show that for any solution $\vct x\in \Rmax^{n}$ of equation $\mtr A\mxsmp \vct x=\vct b$, we have:
\begin{equation}\label{EQN_thm_sparse_recovery_1}
 x_{j}=z_{j},\text{ for all }j\in \supp{z}.
 \end{equation}
 Then, since $\vct{x}^*$ is also a solution we have $x^*_{j}=z_{j}$, for $j\in \supp{z}$. But $\vct{x}^*$ is the sparsest solution. Thus, we necessarily have $x^*_{j}=-\infty$ for $j\not\in \supp{z}$; otherwise, $\vct z$ would be a sparser solution contradicting the assumptions. This implies that $\vct z=\vct x^{*}$.

We now prove~\eqref{EQN_thm_sparse_recovery_1}.
Given a $j\in \supp{z}$, let $i=i(j)\in I$ be the row index such that the condition of the theorem holds. Consider the row index sets $I_j\subseteq I$, $j\in J$ defined in~\eqref{sets_which_cover}. Part  b) of the condition implies that 
\[\bar{x}_{l}=\biginf{t\in I}{}b_{t}-A_{tl}\le b_{s}-A_{sl} <b_{i}-A_{il},\text{ for all }l\in J\setminus\supp{z}.\]
 This implies that the above minimum is not attained at $i$ or:
\begin{equation}\label{EQN_lem_aux_sparse_recovery_1}
i\not\in I_{l},\text{ for all } l\in J\setminus\supp{z}.
\end{equation}
Part a) of the condition implies that 
\[b_{i}=\bigsup{p\in J}{}\paren{A_{ip}+z_p}=A_{ij}+z_{j}.\]
Moreover, by the definition~\eqref{principal_solution} of the principal solution:
\[\bar{x}_{j}=\biginf{q\in I}{}\paren{b_{q}-A_{qj}}\le b_{i}-A_{ij}= z_{j}.\]
But by Theorem~\ref{theorem_multiple_solutions}, only $\bar{x}_{j}=z_{j}$ is possible since the principal solution dominates every other solution.
  Let $k\in \supp{z}$, $k\neq j$ be another index in the support of $\vct z$. We can similarly show that $\bar{x}_{k}=z_{k}$. Now, we claim that $i\not\in I_k$.  If we had $i\in I_{k}$, then $z_k=\bar{x}_{k}=b_{i}-A_{ik}$ or by replacing $b_{i}=A_{ij}+z_{j}$: \[A_{ij}+z_{j}= A_{ik}+z_{k},\]  which contradicts the theorem hypothesis $A_{ij}+z_{j}> A_{ik}+z_{k}$. Thus:
\begin{equation}\label{EQN_lem_aux_sparse_recovery_2}
i\not\in I_{k},\text{ for all } k\in \supp{z}\setminus\set{j}.
\end{equation} 
Since the system $\mtr A\mxsmp \vct x=\vct b$ is solvable, from \eqref{EQN_lem_aux_sparse_recovery_1}, \eqref{EQN_lem_aux_sparse_recovery_2} $j$ is the unique index such that $i\in I_{j}$.  
Hence, set $I$ cannot be covered without including set $I_{j}$ in the covering. By Theorem \ref{theorem_multiple_solutions}, any solution $\vct x\in \Rmax^{n}$ must necessarily have $x_{j}=\bar{x}_{j}=z_j$.
\hfill \qed

\bibliographystyle{spbasic} 
\newcommand{\noopsort}[1]{} \newcommand{\printfirst}[2]{#1}
\newcommand{\singleletter}[1]{#1} \newcommand{\switchargs}[2]{#2#1}

\end{document}